\documentclass[11pt,fleqn]{amsart} 

\usepackage{amsthm}
\usepackage{amsfonts}
\usepackage[english]{babel}
\usepackage{xcolor}
\usepackage{graphicx}
\usepackage{soul}
\usepackage{stfloats}
\usepackage{morefloats}
\usepackage{cite}
\usepackage{lscape}
\usepackage{epstopdf}
\usepackage{braket}
\usepackage[lite]{amsrefs}
\usepackage{mathbbol}
\usepackage{tikz-cd}

\usepackage{amsmath,amstext,amsopn,amsfonts,eucal,amssymb}
\usepackage{graphicx,wrapfig,url}

\newcommand\N{{\mathbb N}}

\newcommand\Z{{\mathbb Z}}

\newcommand\R{{\mathbb R}}

\newcommand\Sph{{\mathbb S}}

\newtheorem{theorem}{Theorem}[section]
\newtheorem{corollary}[theorem]{Corollary}
\newtheorem{lemma}[theorem]{Lemma}
\newtheorem{proposition}[theorem]{Proposition}
\newtheorem{definition}[theorem]{Definition}

\newtheorem{example}[theorem]{Example}

\newtheorem{remark}[theorem]{Remark}

\newtheorem{question}[theorem]{Question}

\newtheorem*{remark*}{Remark}
\newtheorem*{question*}{Question}
\newtheorem*{problem*}{Problem}
\newtheorem*{example*}{Example}

\newcommand{\cal}{\mathcal}

\begin{document}

\title{Planar Equivalence of Knotoids and Quandle Invariants}

\author{MOHAMED ELHAMDADI} 
\address{Department of Mathematics and Statistics, 
	University of South Florida, Tampa, FL 33620, USA} 
\email{emohamed@usf.edu}

\author{WOUT MOLTMAKER} 
\address{Korteweg-de Vries Institute,
    University of Amsterdam, Amsterdam, 1098XG, Netherlands} 
\email{w.c.moltmaker@uva.nl} 

\author{MASAHICO SAITO} 
\address{Department of Mathematics and Statistics, 
	University of South Florida, Tampa, FL 33620, USA} 
\email{saito@usf.edu} 

\begin{abstract}
While knotoids on the sphere are well-understood by a variety of invariants, knotoids on the plane have proven more subtle to classify due to their multitude over knotoids on the sphere and a lack of invariants that detect a diagram's planar nature. In this paper, we investigate equivalence of planar knotoids using quandle colorings and cocycle invariants. These quandle invariants are able to detect planarity by considering quandle colorings that are restricted at distinguished points in the diagram, namely the endpoints and the point-at-infinity. After defining these invariants we consider their applications to  symmetry properties of planar knotoids such as invertibility and chirality. Furthermore we introduce an invariant called the triangular quandle cocycle invariant and show that it is a stronger invariant than the end specified quandle colorings.
\end{abstract}

\maketitle

\date{\empty}


\section{Introduction}\label{sec:Intro}

Knotoids were introduced by V. Turaev \cite{Turaev} and are a generalization of knot diagrams, modelled on the unit interval rather than the circle.
This means that a knotoid diagram is allowed to have two end-points; what distinguishes knotoids from long knots is that the endpoints are allowed to lie anywhere in the diagram, including interior regions. What makes them nontrivial is that we impose an equivalence relation on knotoid diagrams that prevents us from moving any arcs over or under these endpoints, so that we may imagine them as fixed in place on the plane by two infinitely long rods or `rails' \cite{rail}.

Knotoids were originally introduced to study knots \cite{Turaev}, but since have also received considerable research interest in their own right. Several constructions of strong invariants for knots have been extended to knotoids \cite{newinvariants,moltmaker,winding}, and novel invariants of knotoids have been constructed \cite{barbensi,winding}. This work has culminated in a first classification of knotoids for low crossing number \cite{systematic}.

Since this initial investigation, knotoids have mainly found applications to protein topology, and more generally to studying the topology of physical systems of open-ended filamentous structures. We model these using `open curves', i.e.~smooth embeddings $[0,1]\hookrightarrow \mathbb{R}^3$. In particular, the following approach was taken in \cite{Eleni} to construct a Jones polynomial of open curves: given an open curve $C$ in $\mathbb{R}^3$, for a unit vector $v$ in $S^2$ define $C_v$ to be the knotoid diagram obtained from projecting $C$ onto the plane normal to $v$. Then for all but a measure zero subset of $S^2$, $C_v$ is a well-defined knotoid diagram. The idea is then that we can `measure' the entanglement of $C$ by considering the multiset $\{C_v\}_{v\in S^2}$. For any knotoid type $K\in \{C_v\}_{v\in S^2}$ we consider the area $A_K$ of the subset of $S^2$ that is sent to $K$ by the map $C\mapsto [C_v]$. The Jones polynomial of an open curve is then defined as the (a priori infinite) sum of Jones polynomials $V_K(t)$ of all knotoid types $K\in \{C_v\}_{v\in S^2}$, weighted by the areas $A_K$. This convergent sum can be written as an integral as follows, noting that the integrand is well-defined on all but a measure zero subset of $S^2$ \cite{Eleni}:

$$ 
    V_C(t) = \frac{1}{4\pi^2} \int_{v\in S^2} V_{C_v}(t) dA.
$$ 
This Jones polynomial of open curves is not an isotopy invariant of open curves, indeed all open curves are ambient isotopic to one another, but rather it is a `measure' for the open curve's knottedness.
Other measures of entanglement are obtained from following the same procedure using a knotoid invariant other than the Jones polynomial:  as long as the invariant takes values in some $\mathbb{Z}$-module $M$ we obtain a knot measure taking values in $M\otimes \mathbb{R}$. For example see \cite{Eleni2} where $\mathbb{R}$-valued entanglement measures are derived from $\mathbb{Z}$-valued knotoid invariants.

In all applications of this kind considered so far, the knotoid diagrams are thought of as lying on a sphere, rather than their plane projection. 
This can of course be done by adding a point at infinity to the plane. However, unlike for knots and links, for knotoids the equivalence classes on the plane differ from those on the sphere. Clearly adding a point at infinity to the plane is a surjection from knotoids on the plane to knotoids on the sphere, but this map is not injective. In fact it is many-to-one: for example, there are 8 equivalence classes of knotoids with crossing number 4 on the sphere, but 154 on the plane \cite{systematic}.

This observation implies that if one employs an entanglement measure that relies on an invariant of knotoids on the plane rather than on the sphere, then we can obtain a measure with a `higher resolution' for entanglement. This is in the sense that the underlying invariant can distinguish more different knotoids and hence different `kinds of knottedness'. The problem with this approach is that currently, not very many invariants of knotoids on the plane are known. Most straight-forward generalizations of knot and link invariants to knotoids turn out to factor through the one-point compactification of the plane, and hence yield invariants of knotoids on the sphere.

In this paper we define several new planar knotoid invariants based on quandles. Quandles were introduced to knot theory independently by D. Joyce \cite{Joyce} and S. Matveev \cite{Matveev}, and axiomatize the notion of `coloring' a knot diagram, where the quandle acts as a pallet of colors to use. The simplest quandle invariants 
count the number of valid colorings of a diagram, yielding an $\mathbb{N}$-valued invariant. Extending these quandle-coloring counting invariants to planar knotoids, we obtain an $\mathbb{R}$-valued planar entanglement measure for open curves, which enjoys the increased resolution of employing a planar invariant as well as the convenience for applications of taking values in $\mathbb{R}$. Further extending these invariants to quandle homology invariants yields planar knotoid measures taking values in more complicated spaces.

Quandle invariants of knotoids were first studied in \cite{GN}, but the invariants described there do not tend to detect the planar character of knotoids on the plane, and are invariant under one of the `forbidden moves' for knotoids; namely that which pulls an arc over an endpoint. In fact, this invariance was already noted in a different context in \cite{CKS-Geom}. In this paper we give a refinement of these invariants that does detect planarity. We note that a recent article \cite{GR} constructs similar refinements using {\it fundamental} and {\it pointed} quandles; the work presented here is independent 
 and as such we do not refer to these constructions, although the colorings here can be interpreted as quandle homomorphisms from pointed fundamental quandles.

The organization of this article is as follows.  In Section \ref{sec:pre} we briefly give preliminary descriptions of knotoids and quandles. In Section \ref{sec:invariants} we construct quandle invariants of planar knotoids, and discuss how well these invariants can distinguish inequivalent planar knotoids that become equivalent on the sphere. In Section \ref{sec:invertibility} we consider our invariants in relation to planar invertibility. We refine the basic quandle invariants to quandle cocycle invariants in Section \ref{sec:cocycle}, and discuss these in relation to planar chirality in Section \ref{sec:chirality}. Finally in Section \ref{sec:triangular} we describe another quandle cocycle invariant for planar knotoids, called the `triangular' quandle 2-cocycle invariant.

\section{Preliminaries of knotoids and quandles}\label{sec:pre}

In this section we review materials used in this paper.

\subsection{Knotoids}

\begin{definition}
{\rm
Let $\Sigma$ be a surface. A \textit{knotoid diagram} on $\Sigma$ is a smooth immersion $K:[0,1]\to \Sigma$ whose only singularities are transversal self-intersections that are decorated with over/under-crossing data. The images of $0$ and $1$ are called the {\it initial point} (or \textit{tail}) and {\it terminal  point} (or \textit{head}) of the knotoid, respectively. Knotoid diagrams are oriented from tail to head. Two knotoids are said to be \textit{equivalent} if they can be related by a sequence of isotopies of $\Sigma$ and Reidemeister moves; see Figure \ref{Reid}. A \textit{knotoid} on $\Sigma$ is an equivalence class of knotoid diagrams on $\Sigma$.
}
\end{definition}

\begin{figure}[htb]
\begin{center}
\includegraphics[width=3.7in]{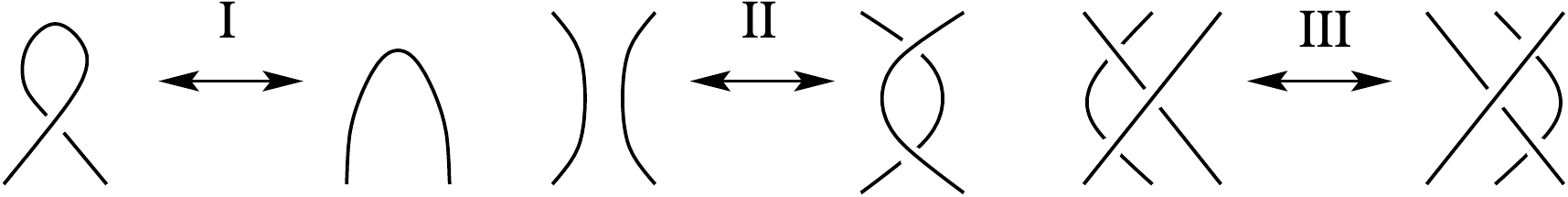}
\end{center}
\caption{}
\label{Reid}
\end{figure}

Note that the Reidemeister moves are not allowed to involve any endpoints of a knotoid, and therefore the \textit{forbidden moves} on a knotoid diagram, depicted in Figure \ref{forbidden} are explicitly not allowed. Note that if these forbidden moves were allowed, all knotoids on any surface $\Sigma$ would be equivalent to the trivial knotoid that has no crossings.

\begin{figure}[htb]
\begin{center}
\includegraphics[width=2.7in]{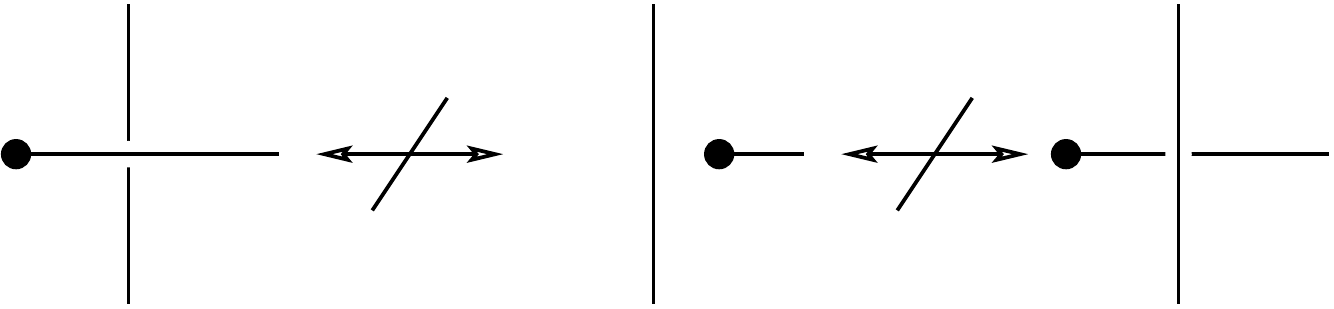}
\end{center}
\caption{}
\label{forbidden}
\end{figure}

In this work we restrict our attention to the cases $\Sigma = S^2$ and $\Sigma = \mathbb{R}^2$. Knotoids on $S^2$ are also called \textit{spherical knotoids}, and knotoids on $\mathbb{R}^2$ are called \textit{planar knotoids}.

Clearly there is a surjection from planar knotoids to spherical knotoids, given by the one-point compactification of the plane. This surjection acts on planar knotoid diagrams by adding a point at infinity; this operation preserves the equivalence of diagrams. For knots the same map exists, but is a bijection since an isotopy carrying an arc of a knot diagram across the added point at infinity can also be achieved by moving the same arc in the other direction, over the entire knot diagram. For knotoids this is not the case, as the endpoints get in the way of an arc moving over the entire diagram. As a result, for knotoids this surjection is not injective. In other words: there are many inequivalent planar knotoids that become equivalent when drawn on the sphere, the only difference between them being that they have different regions as their exterior region.

Conversely, given a spherical knotoid we can turn it into a planar knotoid diagram (non-canonically) by choosing a region in the diagram and removing a point in that region from $S^2$, i.e. designating it as the point at infinity. As such we now introduce the following notation:
Let $K$ be a spherical  knotoid diagram. 
Let $R$ be a region in the diagram. 
The planar knotoid of $K$ when $R$ is specified as the region at infinity 
is denoted by $(K, R)$. 
The spherical and planar equivalence of knotoids are denoted by $\cong_{\Sph^2}$ and $\cong_{\R^2}$, 
respectively.

Spherical knotoids have a natural product operation given by concatenation: for any spherical knotoid diagrams $K_1$ and $K_2$ we may assume without loss of generality that the head of $K_1$ and the tail of $K_2$ lie in the exteriors of their respective diagrams. The product $K_1\cdot K_2$ is then defined by identifying the head of $K_1$ with the tail of $K_2$ in the disjoint union $K_1 \sqcup K_2$. 

For planar knotoids we define the product operation analogously, as long as the exterior region of $K_2$ is appropriate:

\begin{definition}
Let $(K_1,R_1)$ and $(K_2,R_2)$ be planar knotoids such that the tail of $K_2$ lies in $R_2$. Then the \textit{product} $K_1\cdot K_2$ is the planar knotoid diagram obtained by placing a small copy of $K_2$ next to the head of $K_1$ and identifying the tail of $K_2$ with the head of $K_1$.
\end{definition}

\subsection{Quandle invariants}

We start by reviewing the definition of a quandle and give a few examples.
\begin{definition} \cite{Joyce, Matveev}
{\rm
A quandle $(X,*)$ is a set with a binary operation $(x, y) \mapsto x * y$ satisfying the following three axioms:

(I) For all $x \in X$,
$x* x =x$.

(II) For all $y,z \in X$, there exists a unique $x \in X$ such that 
$x*y=z$.

(III) 
For all $x,y,z \in X$, we have
$ (x*y)*z=(x*z)*(y*z). $
}
\end{definition}
Recall that these three axioms correspond respectively to Reidemeister moves of type I, II, and III.

A function $\phi: (X,*) \rightarrow  (Y,\triangleright)$ between quandles  is called a {\it homomorphism}
if $\phi(x \ast y) = f(x) \triangleright f(y)$ 
for any $x,y \in X$.  The notions of isomorphisms and automorphisms are defined naturally. 
The {\it inner} automorphism group of a quandle $X$ is the subgroup of its automorphism group generated by 
the right action defined by the quandle operation.
A quandle is called {\it connected} if its inner automorphism group acts transitively.
Throughout this paper we assume that all quandles are connected, as it is sufficient to focus on them in terms of colorings of knotoid diagrams. 

The following are a few examples of quandles.

\begin{itemize}
\item
Any set $X$ with the operation given by $x*y=x$ for any $x,y \in X$ is
a quandle called the {\it trivial} quandle.
\item
The binary operation of conjugation $x*y=y^{-1}xy$ defines a quandle structure on any group $G$.  Furthermore, for the abelian cyclic group $G=\mathbb{Z}_n$, the binary operation can be written as $x*y=-x+2y$.  This quandle is called the {\it dihedral quandle} and denoted ${\bf R}_n$. 

\item
Let $\Lambda$ be $\mathbb{Z}[t^{\pm 1}]$ or $\mathbb{Z}_n[t^{\pm 1}]$.  Then any $\Lambda$-module $M$
is a quandle with
$x*y=tx+(1-t)y$, $x,y \in M$, called an {\it  Alexander  quandle}.
\end{itemize}


A  {\it coloring}   
of an oriented  knot or knotoid diagram is a
function ${\mathcal C} : R \rightarrow X$, where $X$ is a fixed 
quandle
and $R$ is the set of over-arcs in the diagram satisfying, at each crossing, the conditions (A) and (B) of Figure~\ref{qcolors}. 
The set of colorings of a knotoid diagram $K$ by a quandle $(X,*)$ is denoted by 
${\rm Col}_X(K)$.

Let $D \subset \Sigma$ be a diagram of a knotoid.  The connected components of $\Sigma \setminus D$ are called the regions.  A coloring of the regions is defined as a map from the set of regions of a colored knotoid diagram to a quandle that satisfy the conditions (C), (D) and (E) of Figure~\ref{qcolors}.

\begin{lemma}[\cite{CKS-Geom}] \label{lem:regioncolor}
Let $\cal C \in {\rm Col}_X(K)$ be a coloring of the set of arcs of a knot diagram $K$ by a connected quandle $(X,*)$.
For a given color $x$ of a particular region $R$, 
 there is a unique region coloring $\tilde{\cal C}$ that extends $\cal C$.
\end{lemma}

The proof of this lemma is to define region colors along arcs that go under the knot diagram and determine colors
according to the rule. The uniqueness follows from independence under homotopy of such arcs, as shown in Figure~\ref{qcolors} (E), where the rightmost region color is independent of choice of arcs that goes over and under the center crossing.
This lemma can be extended to knotoids in the following manner.

\begin{lemma}\label{lem:color}
Let $\cal C \in {\rm Col}_X(K)$ be a coloring of the set of arcs of a spherical knotoid 
diagram $K$ by a connected quandle $(X,*)$.
Then
 there is a unique region coloring $\tilde{\cal C}$ that extends $\cal C$.
\end{lemma}

\begin{proof}
Let $p_\iota$ and $p_\tau$  be the initial and terminal  points of $K$.
Let $\gamma$ be an arc from  $p_\iota$ to  $p_\tau$ that goes under the arcs of $K$.
Let $\hat K$ be the knot diagram given by the union $K \cup \gamma$, called an {\it under closure} of $K$.
Extend the coloring ${\cal C}$ of $K$ along $\gamma$ as it goes under each arc of $K$, following the coloring
rule. This argument is also used in \cite{Turaev}.
When it reaches near $p_\tau$, the colors of these points coincide.
This is seen by an argument showing that the end colors of (1-1)-tangles coincide, see \cite{CDS,Jozef}.
Independence of the choice of $\gamma$ can be seen by the homotopy argument as in Lemma~\ref{lem:regioncolor}.
\end{proof}


\begin{figure}[htb]
\begin{center}
\includegraphics[width=5.5in]{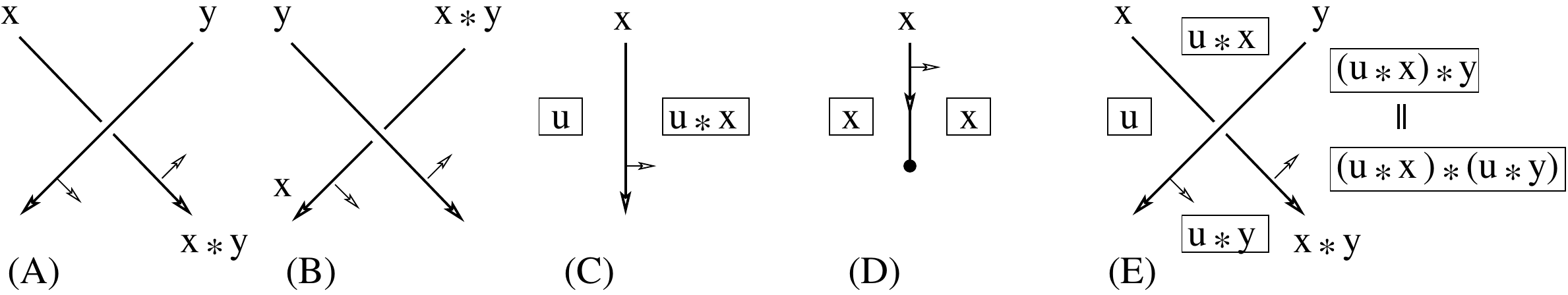}
\end{center}
\caption{}
\label{qcolors}
\end{figure}

\section{End specified triangular colorings}\label{sec:invariants}
              
Let $K$ be a spherical oriented knotoid diagram.
Let $R_\infty$ be the region specified as the region at infinity.
Let $R_\iota$ and $R_\tau$ be the regions at the initial and terminal end points, respectively.
Let  $p_\infty$ be a point in $R_\infty$, and  $p_\iota$, $p_\tau$ 
be the initial and terminal end points,  respectively.

Let $(X,*)$ be a quandle.
Denote by $\Z[X]$ the free integral module over $X$, consisting of elements of the formal
finite sum $\sum  a_x x$, where $a_x\in \Z$ and $x \in X$.
The addition is defined monomial wise and extended linearly, $( \sum a_x x ) + ( \sum b_y y)=\sum (a_z + b_z ) z $.
Since $\Z [X \times X]\cong \Z [X] \otimes \Z[X]$, we denote their elements by $ \sum a_{(x,y)}  x \otimes y $,
and similar for multiple tensors.

\begin{definition}
{\rm
Let $(X,*)$ be a quandle and let $(K, R_\infty)$ be a planar oriented knotoid.
An ordered  {\it quandle triple} $(a,b,c) \in X^3$ is an ordered triple such that 
$a,b,c$ belong to the same connected component of $X$.

An {\it end specified coloring} ${\cal C}: {\cal A}\cup {\cal R} \rightarrow X$
with respect to a quandle triple $(a,b,c)$  is a coloring
of arcs and regions of $K$ by $X$ such that 
$({\cal C} (p_\infty) , {\cal C} (p_\iota) , {\cal C} (p_\tau)  )=(a,b,c) $.
 The set of all end specified colorings with respect to a quandle triple $(a,b,c)$ is denoted by 
${\rm Col}^{(a,b,c)}_X (K, R_\infty)$, and its cardinality is denoted by 
$ | {\rm Col}^{(a,b,c)}_X (K , R_\infty ) | $.
The  {\it end specified coloring invariant} of $(K,R)$ with respect to $X$ and $(a,b,c)$ 
is a formal sum 
$$ {\cal Col}_X (K, R_\infty)  := \sum_{ (a,b,c) \in X^3 } | {\rm Col }_X^{(a,b,c)} (K, R_\infty) | (a \otimes b \otimes c )  \in \Z[X]^{\otimes 3}. $$

Similarly, for a spherical knotoid $K$, an end specified coloring ${\cal C}$ is defined by
a coloring such that $({\cal C}(p_\iota), {\cal C}(p_\tau))=(b,c)$ for $(b,c) \in X^2$.
The set of all end specified coloring with respect to a quandle pair $(b,c)\in X^2$ is denoted by 
${\rm Col}^{(b,c)}_X (K)$.
Similar colorings are defined for planar knot diagrams when no color is specified for the region at infinity.

}
\end{definition}

\begin{lemma}
Let $(K,R)$ be a planar knotoid and $X$ be a quandle.
Then $\sum_{a \in X} | {\rm Col}^{(a,b,c)}(K, R) | = |   {\rm Col}^{(b,c)}(K) |$.
\end{lemma}

\begin{proof}
For a given $(b,c)\in X^2$, let ${\cal C} \in {\rm Col}^{(b,c)}(K)$. 
Then by Lemma~\ref{lem:color}, there is a unique $a \in X$ such that ${\cal C} \in {\rm Col}^{(a,b,c)}(K, R)$.
Hence the statement follows.
\end{proof}

\begin{figure}[htb]
\begin{center}
\includegraphics[width=2.5in]{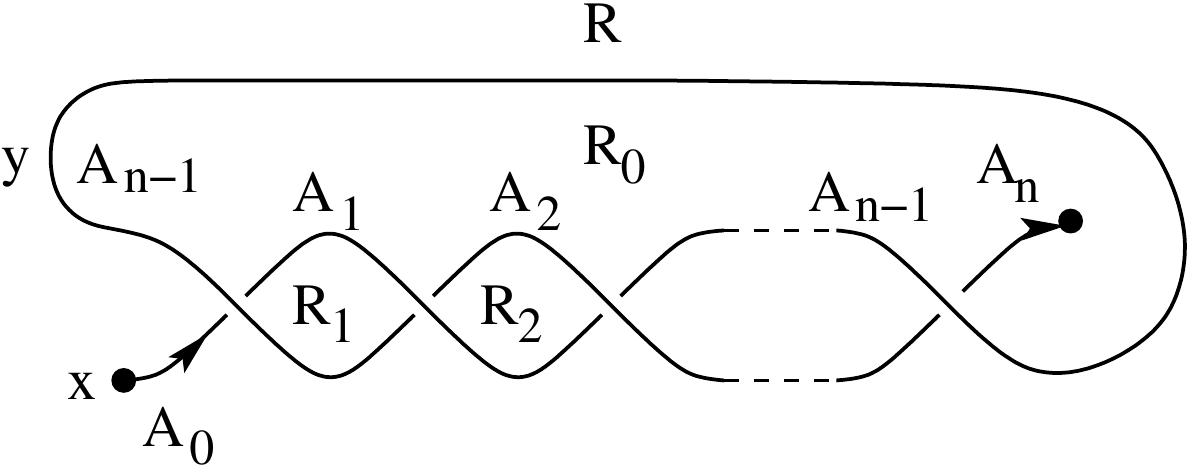}
\end{center}
\caption{}
\label{T2n}
\end{figure}

\begin{example}\label{ex:Tn}
{\rm
In Figure~\ref{T2n}, a knotoid $T_n$ is depicted, where  $n=2 h$ is an even integer.
Arcs $A_i$ and regions $R_j$ ($i =0, 1, \ldots, n$, $j =0, 1, \ldots, n-1$) are indicated, as well as the region at infinity $R$.
In this example we show that 
$(T_n, R_1) \not \equiv_{\R^2} (T_n, R_2)$
using colorings.

We examine possible colorings. 
Let ${\cal C}(p_\iota)=x$ and ${\cal C}(A_{n-1})=y$.
Recall that an Alexander quandle $X$ has the operation  $u*v=tu+(1-t)v$
for $u,v \in X$.
If $T_n$ is colored by an Alexander quandle $X$ such that  ${\cal C}(p_\iota)=x$ and ${\cal C}(A_{n-1})=y$,
then one inductively computes that 
$${\cal C} (A_k)=t \xi_{k-1}(t) x + \xi_k (t) y, \ {\rm  where} \
\xi_k (t) =1 - t + t^2 - \cdots + (-t)^k ; \ (\xi_0(t)=1, \ \xi_1(t)=1-t). $$
Similarly, we have 
\begin{eqnarray*}
{\cal C} (R_1) &=&  tx+(1-t)y, \\
{\cal C} (R_2) &=&   (2-t)tx + (1-t)^2 y,\\
{\cal C} (R_k) & = & [ 1 + (1-t)\xi_{k-2} (t) ] tx + (1-t) \xi_{k-1} (t)    y \ {\rm for } \ k>2.
\end{eqnarray*}

For a dihedral quandle ${\bf R}_\ell = \Z_\ell [t] /(t+1)$, we have 
$\xi_k(-1)=k+1$, ${\cal C} (A_k)= - k x + (k+1) y$ and ${\cal C} (R_k)= (1-2k) x + (2k)y $ for $k>0$
and ${\cal C}(R_0)=-x+2y$. 
A  pair $(x,y)$  gives rise to a  coloring if and only if ${\cal C} (A_{n-1})=y$ and 
${\cal C} (R_0)= {\cal C} (A_n)$, that is, 
$$-(n-1) x + n y = y \quad {\rm and } \quad -x+2y = -n x + (n+1) y , $$
 the same condition. 
This condition is satisfied if and only if $n=1$ mod $\ell$.
Thus we consider the case $n=\ell +1$, and take $X={\bf R}_{n -1}= {\bf R}_{2h  -1}$.
Then there is a coloring of $K$ for any $x, y \in X$. 
In this case we also have ${\cal C} (R_0)= {\cal C} (A_n)=2y-x$. 
Since $2$ is invertible, the correspondence $(x,y) \mapsto (x, 2y-x)$ is bijective.
Since ${\cal C}(R_1)=2y-1$ as well, for
$x=0$ and $y=1$, we have
$|{\rm Col}^{(2,0,2)}(T_n, R_1)|=1$.
Since ${\cal C}(R_2)=4y-3x=4$ for $x=0$ and $y=1$,
we have $|{\rm Col}^{(2,0,2)}(T_n, R_2)|=0$.
Hence we obtain that $(T_n, R_1)$ and $(T_n, R_2)$ are not planarily equivalent.
}
\end{example}

\begin{proposition}\label{prop:infinite}
For any positive integer $m$, there exists a spherical knotoid $K$ that has $m$ regions $R_i$, $i=1, \ldots, m$, 
such that the planar knotoids $(K, R_i)$ 
are pairwise inequivalent as planar knotoids.
\end{proposition}

\begin{proof}
We set $K$ to be $T_n$ in Example~\ref{ex:Tn}, where $n \geq m+2$.
We  use the notations as in Figure~\ref{T2n}, and computations in Example~\ref{ex:Tn}
with the dihedral quandle ${\bf R}_\ell$, 
 where $n=\ell + 1$,
 $\ell = n-1=2h -1$.

For a fixed $k$, $1 \leq k \leq m$, 
we consider the quandle triple  $(2k, 0, 2)$, where $2k$ 
is regarded as an element of $\Z_n$. 
Then as in the example, we have 
$$| {\rm Col}^{(2k,0,2)}_X(K, R_k) | = 1, \ {\rm and}\ 
| {\rm Col}^{(2k,0,2)}_X(K, R_{k'}) | = 0$$
 for all $k' \neq k$, $k'=1, \ldots, m$. 
By varying $k=1, \ldots , m$, we obtain that 
$(K, R_i)$ are pairwise inequivalent.
\end{proof}

\begin{example}\label{ex:Xi}
{\rm
To compute the number of colorings over all possible quandle triples for $(T_n, R_k)$,
we use two colors $x, y$ in Figure~\ref{T2n}  as parameters. 
To have nontrivial colorings, we set $X={\bf R}_\ell$ with $\ell = n-1$ as computed in Example~\ref{ex:Tn}.
We have computed that for a given $(x,y)$, we have 
$${\cal C} (p_\iota )=x, \ {\cal C} (p_\tau) =  {\cal C} (A_n)= {\cal C} (R_0) = 2 y - x , \  {\cal C} (p_\infty)= {\cal C} (R_k) =  (1-2k)x + (2k)y. $$
For a given $(x,y,k)$, we have 
$$| {\rm Col}^{((1-2k)x + (2k)y, x, 2y-x)}_X(K, R_k) | = 1, $$
and for other $(a,b,c)$ we have $| {\rm Col}^{(a,b,c)}_X(K, R_k) | = 0$.
Therefore we obtain 
$${\cal Col}_X (K,R_k) = \sum_{ (a,b,c) \in Qt(X) } | {\rm Col}^{(a,b,c)}_X(K, R_k) |  (a \otimes b \otimes c) =
\sum_{ x, y \in X } ( [(1-2k)x + (2k)y] \otimes  x \otimes [ 2y-x] ).$$
}
\end{example}

\begin{definition}
{\rm
The {\it matching product} on $\Z  [X]^{\otimes 3} $, 
$$\vartriangle : ( \Z  [X]^{ \otimes 3 } ) \times ( \Z  [X]^{ \otimes 3 } ) \rightarrow \Z  [X]^{ \otimes 3 }  , $$
is defined on monomials by 
$$  a_{x,y,z}  (x  \otimes y \otimes z )  \vartriangle b_{u,v,w}  (u  \otimes  v \otimes w ) = \begin{cases}
(a_{x,y,z}  b_{u,v,w}  ) ( x \otimes y \otimes w ) & {\text {if }}  z=u=v \\
0 & {\text{otherwise} }
\end{cases}
$$
and extended linearly.
}
\end{definition}

\begin{lemma}\label{lem:prod}
Let $(X, *)$ be a quandle.
Let $(K_1,R_1)$, $(K_2,R_2)$ be planar knotoids,
and $p_\iota^i$ and  $p_\tau^i$ be the initial and terminal points of $K_i$ for $i=1,2$.
Suppose  $K_2$  satisfies $p_\iota^2 \in R_2$ so that the product $K_1\cdot K_2$ is defined.
Then we have 
$$  {\cal Col}_X (K_1 \cdot K_2 ,R_1) = {\cal Col}_X (K_1 ,R_1 )  \vartriangle  {\cal Col}_X (K_2 ,R_2 )  . $$

\end{lemma}

\begin{proof}
This is checked monomial wise.
Then  the initial and terminal points $p_\iota$ and  $p_\tau$ of $K_1 \cdot K_2$ are 
$p_\iota=p_\iota^1$ and $p_\tau=p_\tau^2$, respectively.
The other points $p_\tau^1$ and $p_\iota^2$ are identified in $K_1 \cdot K_2$.
For $a,b,c,d \in X$,  ${\cal C}_1 \in  {\rm Col}^{(a,b,c)}_X(K_1)$ and  ${\cal C}_2 \in  {\rm Col}^{(c,c,d)}_X(K_2)$, 
a coloring ${\cal C}={\cal C}_1 \cdot {\cal C}_2  \in  {\rm Col}^{(a,b,d)}_X(K_1\cdot K_2)$,
that restricts to ${\cal C}_1$ and $ {\cal C}_2$ on each factor,  is well defined after
identifying $p_\tau^1$ and $p_\iota^2$, since ${\cal C}_1 (  p_\tau^1 ) =     {\cal C}_2 (p_\iota^2) $. 
Hence we obtain 
$$ | {\rm Col}^{(a,a,d)}_X(K_1\cdot K_2) |= | {\rm Col}^{(a,b,c)}_X(K_1) | |  {\rm Col}^{(c,c,d)}_X(K_2)|. $$ 
When ${\cal C}_1 (  p_\tau^1 ) \neq     {\cal C}_2 (p_\iota^2)$, the colorings do not extend.
Thus  there is a bijection between the set of colorings of $K_1$ and $K_2$ such that 
${\cal C}_1 (  p_\tau^1 ) =  {\cal C}_2 (p_\iota^2)$ and the set of colorings of $K_1\cdot K_2$,
and the equality holds for monomials.
\end{proof}

\begin{example}\label{ex:Tnm}
{\rm
Let $T_n$ be the knotoid in  Example~\ref{ex:Tn}, where it was shown that 
${\cal Col}_X(T_n,R) = \sum_{a, b \in {\bf R}_{2h-1}} a \otimes  a \otimes b$ for a positive integer $h$, $n=2h$. 
Lemma~\ref{lem:prod} implies that 
$${\cal Col}_X(T_n \cdot T_n,R) ={\cal Col}_X(T_n,R) \vartriangle {\cal Col}_X(T_n,R) =
(\sum a \otimes a \otimes b ) \vartriangle (\sum c \otimes c \otimes d ) . $$
See Figure~\ref{two} for a diagram of $T_n \cdot T_n$. 
For any $(a, d)\in {\bf R}_{2h-1}$, there are $2h-1=n-1$ elements $b=c \in  {\bf R}_{2h-1}$
such that $(a \otimes a \otimes b) \vartriangle (c \otimes c \otimes d)=(a, a, d)$, 
hence $$(\sum a \otimes a \otimes  b ) \vartriangle (\sum c \otimes c \otimes  d ) = \sum (n-1) (a \otimes a \otimes  d) . $$
Thus we have ${\cal Col}_X(T_n \cdot T_n) = \sum_{a,d \in {\bf R}_{n-1} } (n-1) (a \otimes  a \otimes d) $.

}
\end{example}

\begin{corollary}
For any positive  integers $m, \ell$, there exist a knotoid $K$ and a quandle $(X, *)$ such that 
${\cal Col}_X(K)$ consists of more than $m$ monomials whose coefficients are all larger than $\ell$.
\end{corollary}

\begin{proof}
Repeated applications of the computations given in  Example~\ref{ex:Tnm}
show that ${\cal Col}_{ {\bf R}_{n-1} } ( ( T_n)^{\cdot k}  )$, where the knotoid is the $k$-fold knotoid product, 
is equal to $\sum_{a,b \in {\bf R}_{n-1} } (n-1)^{k-1} a \otimes b $.
Thus sufficiently large choices of $n$ and $k$ for given $m$ and $\ell$ give the result.
\end{proof}

\begin{question}\label{q:values}
{\rm
What are the possible values of the invariant?
There are more specific questions that can be raised, such as:
 Is it true that for any positive integer $n$, there exist a knotoid $K$ and a quandle $(X,*)$
such that ${\cal Col}_X(K)$ consists of $n$ nonzero monomials?
Are there any relations among coefficients of nonzero monomials?
}
\end{question}

\section{Planar equivalence and invertibility}\label{sec:invertibility}

In this section we investigate relations between planar equivalence and invertibility.
For a planar or spherical 
knotoid  $K$, we denote by $rK$ the knotoid   obtained from $K$ by reversing its orientation. 
If a planar knotoid $(K, R)$ is planarily equivalent to  $(rK, R)$, then it is called 
{\it (planarily) invertible}. Otherwise it is called {\it (planarily) noninvertible}.
Spherical invertibility is defined analogously, but without reference to a region at infinity.

\begin{example}
{\rm
Let $T_n$ be a knotoid depicted in Figure~\ref{T2n}, $n>0$. 
Let $R$ be the region at infinity.
Then $(rT_n , R)$ is not equivalent to $(T_n, R)$.
This is because for $(T_n,R)$, $p_\infty$ and $p_\iota$ both belong to $R$, but $p_\tau$ does not,
while for $(rT_n, R)$, $p_\infty$ and $p_\tau$ both belong to $R$, but $p_\iota$ does not.

However, $T_n$ and $rT_n$ are spherically equivalent, as the outer arc can be isotoped through infinity,
so that $T_n$ is spherically invertible.
}
\end{example}

\begin{remark}
{\rm 
The argument in the preceding example motivates to define the following concept of  distances.
Let $\gamma_{\iota}$, $\gamma_{\tau}$ and $\gamma_{\iota\tau}$ 
be oriented paths from $p_\infty$ to  $p_\iota$, $p_\infty$ to  $p_\tau$ and
$p_\iota$ to $p_\tau$, respectively, all whose intersections with $K$ are transversal crossing points, away from the crossings of $K$.
Assume  that these arcs intersect with $K$ in a finitely many double crossings.

Then the crossing distance  from  the terminal  point $p_\infty$ to the initial point $p_\iota$ 
(resp. the terminal point $p_\tau$)
is the 
minimum number of crossings of $\gamma_{\iota}$ and $K$ among all equivalent diagrams, and denoted by $d_\iota$ (resp. $d_\tau$).
The distance between the initial and terminal points, $d_{\iota\tau}$ was similarly defined as the \textit{height} of $K$ in \cite{GK}.
The distances $d_\iota$ and  $d_\tau$ can be used to detect invertibility.
}
\end{remark}

A quandle is called {\it involutive} if its right action is involutive, i.e.~$(x*y)*y=x$ for all elements $x,y$.
For example the dihedral quandles ${\bf R}_n$ are involutive.
As colorings by involutive quandles are well defined for unoriented diagrams, 
there is a bijection (identity correspondence)  between the sets of colorings before and after the orientation reversal.
Since the reversal changes the initial and terminal end points, we have the following.

\begin{lemma}
Let $X$ be an involutive quandle.
For planar  knotoid $(K, R)$ and for any quandle triple $(a,b,c)$,
we have 
 $$| {\rm Col}^{(a,b,c)}_X(K, R) | = | {\rm Col}^{(a,c,b)}_X(r K, R) | .$$
\end{lemma}

\begin{example}
{\rm
In the proof of Proposition~\ref{prop:infinite},  using the same notation, 
we have $| {\rm Col}^{(2k, 0, 2)}_X (T_n, R_k)|=1$.
Using the computation in Example~\ref{ex:Xi}: 
$${\cal Col}_X (K,R_k) =\sum_{ x, y \in X } ( ( 1-2k)x + (2k)y ) \otimes  x\otimes ( 2y-x) . $$

If $x=2$ and $2y-x=2$, then we obtain $y=2$, then $(1-2k)x + (2k)y=2$.
Hence if $k \neq 1$, then $| {\rm Col}^{(2k, 0, 2)}_X ( r T_n, R_k)|=0$, so that 
$(T_n, R_k)$ and $(rT_n, R_k)$ are not planarily equivalent, and $(T_n, R_k)$ is noninvertible.

}
\end{example}

By varying values of $n$ and $k$ in the preceding example, we obtain the following result.

\begin{theorem}
For any positive integer $m$, there exists a spherically invertible knotoid $K$ with regions $R_i$, $i=1, \ldots, m$, 
such that $(K, R_k)$ are planarily noninvertible for $i=1, \ldots, m$.
\end{theorem}

\section{Planar cocycle invariants with end specified colors}\label{sec:cocycle}

In this section we define quandle cocycle invariants for knotoids in a manner similar to
\cite{CJKLS}. Such invariants have been defined in \cite{Cazet}, and the difference here is to use
end restricted colorings and obtain stronger invariants, in particular of planar knotoids, that distinguish 
properties between planar versus spherical equivalences.
In the following section we use these invariants to study planar chirality.

We recall \cite{CJKLS} that a function $\phi: X \times X \rightarrow A$, where $X$ is a quandle and $A$ is an abelian group, is called a {\it quandle 2-cocycle} if it satisfies 
$$ \phi (x,x)=0 \qquad\text{and}\qquad \phi(x,y) - \phi (x,z) +  \phi(x*y, z) - \phi(x*z, y*z)=0$$
for all $x,y,z \in X$. 
The abelian group  of all quandle 2-cocycles is denoted by $Z^2_{\rm Q}(X,A)$. 
Further details of quandle (co)homology theory can be found in \cite{CKS}.

\begin{definition}
{\rm
Let ${\cal C}: {\cal A}\cup {\cal R} \rightarrow X$ be an  end specified coloring
with respect to a quandle triple $(a,b,c)$  of $K$ by $X$ such that 
$({\cal C} (p_\infty) , {\cal C} (p_\iota) , {\cal C} (p_\tau)  )=(a,b,c) $.

Let $(X,*)$ be a quandle, $(K,R)$ a planar knotoid diagram.
Let $\phi \in Z^2_Q (X, A)$ be a quandle 2-cocycle with a finite  abelian coefficient group $A$.
Let $B_\phi (K, {\cal C} ) $ denote the product of the cocycle values $\phi(x_j, y_j)^{\epsilon(j)}$ over all crossings $j$
with respect to a coloring ${\cal C}$.
Here $\epsilon(j)=\pm1$, depending on the sign of crossing $j$.
Recall that the product is with respect to the group operation of $A$, which is abelian but denoted with multiplicative notation.

The {\it planar quandle 2-cocycle invariant} is defined by 
$$ \Phi_\phi^{(a,b,c)} (K) = \sum_{ C \in {\rm Col}_X^{(a,b,c)} (K, R) } B_\phi (K, {\cal C} ) .
$$
Consider the quandle module $ \Z A [ X] ( = (\Z [A] ) [X] )  $, whose elements consists of
finite sums $\sum a_x x$, $a_x \in \Z[A]$.
Then define 
$$ {\Phi}_\phi (K,R) = \sum \Phi_\phi^{(a,b,c)} (K,R) (a \otimes b \otimes c ) \in  \Z A [ X] ^{\otimes 3}.
$$

Similarly, the 2-cocycle invariants $\Phi_\phi^{(b,c)}(K)$  and $\Phi_\phi (K)$  are defined for spherical knotoids using ${\rm Col}^{(b,c)}(K)$ instead of  ${\rm Col}^{(a, b,c)}(K, R)$.
}
\end{definition}

In a similar manner, end-specified 3-cocycle invariants can be also defined; see \cite{CJKLS} for details on 3-cocycle invariants of knots.

\begin{definition}\label{def:3cocyinv}
{\rm 

Let $\psi \in Z^3_Q(X, A)$ be a quandle 3-cocycle of $X$ with a finite abelian coefficient group $A$.
For a region (and arc) coloring ${\cal C}$,
let $B_\psi (K, {\cal C} ) $ denote the product of the cocycle values $\psi(u_j, x_j, y_j)^{\epsilon(j)}$ over all crossings $j$
with respect to a coloring ${\cal C}$, where $u_j$ is the source color at the crossing.

Define the {\it planar quandle 3-cocycle invariant for knotoids } to be
$$ \Psi_\psi^{(a,b,c)}(K) = \sum_{ C \in {\rm Col}^{(a,b,c)}_X (K) } B_\psi (K, {\cal C}) . 
$$
Then 
define 
$$ {\Psi}_\psi (K) = \sum \Phi_\psi^{(a,b,c)} (K) (a \otimes b \otimes c) \in  \Z A [ X]^{\otimes 3}.
$$

}
\end{definition}

By  routine check of Reidemeister moves as in \cite{CJKLS}, we have that these cocycle invariants are well defined.

\section{Planar equivalence and chirality }\label{sec:chirality}

In this section we discuss differences of chirality of knotoid between spherical and planar equivalences.
For a spherical or planar knotoid diagram $K$, we denote by $mK$ the knotoid obtained from $K$ by reversing all its crossings.

First we apply the quandle 3-cocycle invariant defined in the preceding section to show that there exists 
a knotoid that is achiral spherically but chiral planarily.

Let $K$ be a figure eight knotoid diagram depicted in solid arcs in Figure~\ref{fig8}.
We show that $K'$ is spherically amphicheiral but not planarily.
Let $K'$ be a diagram obtained by replacing the outer solid arc by the dotted arc.
By rotating $K$ by 180 degrees about the midpoint of the two end points, 
we see that $K'$ is equivalent to $rmK$.
Since this replacement of the outer arc is a spherical move, we have that $K'= rmK$, hence $K \cong_{\Sph^2} rmK$.
The diagram $K'$ with the exterior region $R_\infty$ at infinity is the same as $(K, R_1)$, 
so we have $(K', R_\infty) =(rmK, R_\infty)  \cong_{\R^2} (K, R_1)$.

\begin{figure}[htb]
\begin{center}
\includegraphics[width=1.5in]{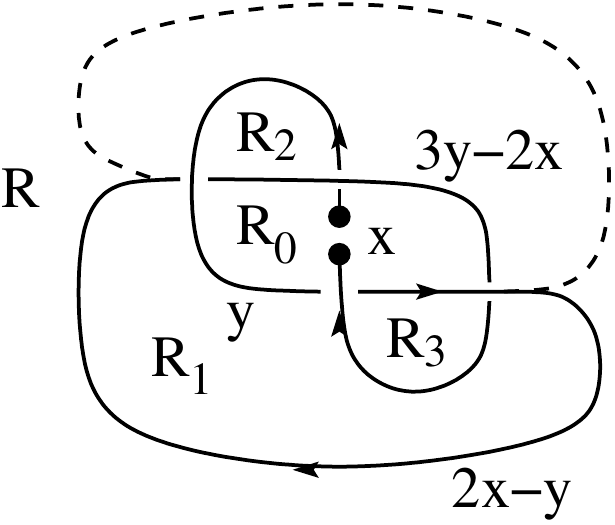}
\end{center}
\caption{}
\label{fig8}
\end{figure}

\begin{proposition}
The figure eight knotoid $K$ as depicted in Figure~\ref{fig8} has the property 
that 
$K \cong_{\Sph^2} rmK$, but a planar diagram $K'$ spherically equivalent to $rmK$ has the property that 
$(K , R) \not \cong_{\R^2} (K', R_\infty)$.
\end{proposition}

\begin{proof}
We have shown that $K \cong_{\Sph^2} rmK$, and below we show $(K, R) \not \cong_{\R^2} (K, R_1)$.
It is well known that $K$ is nontrivially colored by ${\bf R}_5$.
A coloring of arcs are indicated in the figure with variables.
Specifically, an element $x$ of ${\bf R}_5$ is assigned to the end points and the region $R_0$ that they belong to,
and another adjacent arc is colored by $y$. The colors of remaining arcs are determined by crossings, and are as indicated
in the figure. It is checked that the indicated coloring is well defined modulo $5$, and uniquely determined by
a choice of $x, y \in {\bf R}_5$.
The requirement that the color of $R_0$ is $x$ determines all  colors of the other regions uniquely,
and is computed as ${\rm Col}(R)=y$, ${\rm Col}(R_1)=2y-x$, ${\rm Col}(R_2)=y$ and ${\rm Col}(R_3)=3x-2y$.
Thus all colors, on arcs and regions, are uniquely determines by $x, y \in {\bf R}_5$.

If we choose $(p_\infty, p_\iota, p_\tau)=(y, x, x)=(1,0,0)$, for example, we obtain a unique nontrivial coloring and  that 
${\rm Col}^{(1,0,0)} (K, R)=1$. 
On the other hand, if we choose $R_1$ as the region at infinity, then for the triple $(1,0,0)$, we have 
${\rm Col}(R_1)=2y-x=1$ and $x=0$, which implies $y=3$, and we obtain ${\rm Col}^{(1,0,0)} (K, R)=1$ as well.

We apply the cocycle invariant in Definition~\ref{def:3cocyinv}, with the dihedral quandle ${\bf R}_5$ with the 
coefficient group of cohomology $\Z_5$. 
The values of $\Psi_\psi^{(a,b,c)}$ are in the group ring $\Z[\Z_5]$, and written as polynomial in multiplicative generator $u$ of 
$\Z_5$.
There is a nontrivial 3-cocycle in $Z^3_{\rm Q}({\bf R}_5, \Z_5)$. The general case of 3-cocycle for ${\bf R}_p$ with prime $p$ is defined in \cite{Mochi}, 
simplified in \cite{AsamiSatoh,HashiTana} and given by the following formula:
$$\psi(x,y,z)=(x-y) [ (2z-y)^p + y^2 - (2z)^p) ] /p = (x-y) \sum_{i=1}^{p-1} i^{(-1)} (-x)^i (2y)^{p-i} \in \Z_p .$$
The cocycle weights for the top two positive crossings from left to right computes, with given variables,
$\psi (y, 3y-2x, y)$ and $\psi ( y,y, 3y-2x)=0$.
Those bottom two negative crossing from left and right read
$- \psi (x,y,x)$ and $- \psi (y, 3y-2x, 2x-y)$. 
With these cocycle values  we compute that 
$ \Psi_\psi^{(1,0,0)}(K, R)=u^4$, where $u$ is a multiplicative generator of $\Z_5$.

For $K'$, the contributions of crossings are the same, when only the outer arc is changed and labels are kept,
and the region at infinity is changed from $R$ to $R_1$.
Similar computations with ${\rm Col}(R_1)=2y-x=1$ and $x=0$ shows that  $ \Psi_\psi^{(1,0,0)}(K, R_1)=u$. 
Hence we obtain $(K, R) \not \cong_{\R^2} (K, R_1)$. 
\end{proof}

\begin{remark}
{\rm
We conjecture that 
there are infinitely many spherically achiral knotoids that are planarily chiral.
A family of potential candidate is depicted in Figure~\ref{2Bfamily}, that are  generalizations
of the figure eight knotoid, where the center arc is cut into a pair of end points.
(For some integers they have multiple components.)
The argument similar to that in proof of the preceding proposition shows that they are spherically achiral.

We also conjecture that for any positive integer $n$, there is a  planar knotoid diagram $K_n$ with the exterior region $R_0$ at infinity that is planarily achiral $(K, R_0) \cong_{\R^2} (rmK, R_0)$,
such that there are regions $R_k$, $k=1, \ldots, n$ with $(K_n, R_i) \not \cong_{\R^2} (K_n, R_j)$
for $i \neq j$, $i, j = 0, 1, \ldots, n$. Note that in this case all planar knotoids $(K_n, R_i)$ are spherically equivalent,
so that they are all spherically achiral.
A potential example is depicted in Figure~\ref{amphicheiral}. A rotation shows that this diagram is planarily achiral.
This example is motivated by $8_3$ which is achiral. By adding more twists on the top and bottom twists,
we obtain similar candidates.
}
\end{remark}

\begin{figure}[htb]
\centering
\begin{minipage}{0.5\textwidth}
\centering
\includegraphics[width=2.2in]{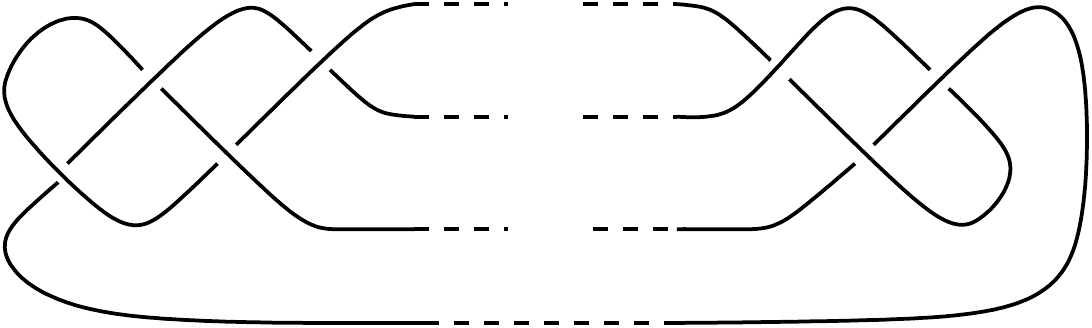}
\caption{}
\label{2Bfamily}
\end{minipage}%
\begin{minipage}{.5\textwidth}
\centering
\includegraphics[width=2.2in]{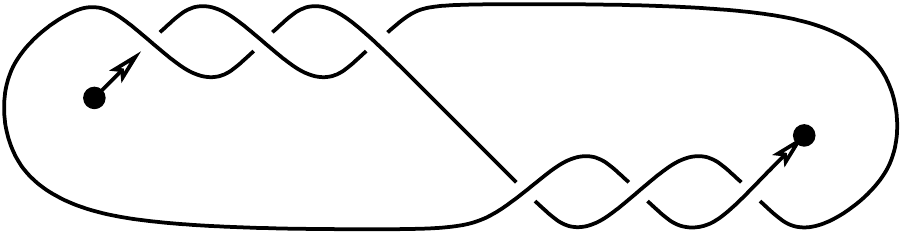}
\caption{}
\label{amphicheiral}
\end{minipage}
\end{figure}


\section{Triangular quandle $2$-cocycle invariant}\label{sec:triangular}

As before, let $K$ be a spherical oriented knotoid diagram,
 $R_\infty$ be the region specified as the region at infinity,
  $p_\infty$ be a point in $R_\infty$, and  $p_\iota$, $p_\tau$ 
be the initial and terminal end points,  respectively.

Let $\gamma_{\iota}$, $\gamma_{\tau}$ and $\gamma_{\iota\tau}$ 
be oriented paths from $p_\infty$ to  $p_\iota$, $p_\infty$ to  $p_\tau$ and
$p_\iota$ to $p_\tau$, respectively, that goes under all arcs of $K$ and away from 
crossings of $K$.  
We may assume  that these arcs intersect with $K$ in finitely many double crossings.
We may  also assume that their self intersections or intersections between them are away from those with $K$ and arcs with $K$.

Let $X$ be a quandle and $\phi$ be its 2-cocycle with coefficient in an abelian group $A$ in multiplicative notation. 
We define the {\it triangular quandle $2$-cocycle invariant} for a knotoid $K$ as follows.
Let $X$ be a quandle, $(a,b,c)$ be a quandle triple, and ${\rm Col}^{(a,b,c)}_X (K)$ be the set of end specified colorings with respect to $(a,b,c)$.
Each arc of the paths $\gamma_h$, $h=\iota, \tau, \iota\tau$ inherit the same color of the region where it is included. Denote the same letter ${\cal C}$ for such a coloring.
Then for each crossing $\lambda \in \gamma_h \cap K$,  the weight $B({\cal C}, \lambda)$
is defined as in \cite{CJKLS} to be $\phi(x(\lambda), y (\lambda) )^{\epsilon (\lambda) } $, where 
$(x(c), y (c) )$ are the source colors at $\lambda$ for ${\cal C}$, and $\epsilon (c) $ denotes the sign of $c$.
Let $B_\phi ({\cal C}, \gamma_h) = \prod_{\lambda \in  \gamma_h \cap K} \phi(x(\lambda), y (\lambda) )^{\epsilon (\lambda) }  \in A$.

\begin{figure}[htb]
\begin{center}
\includegraphics[width=3in]{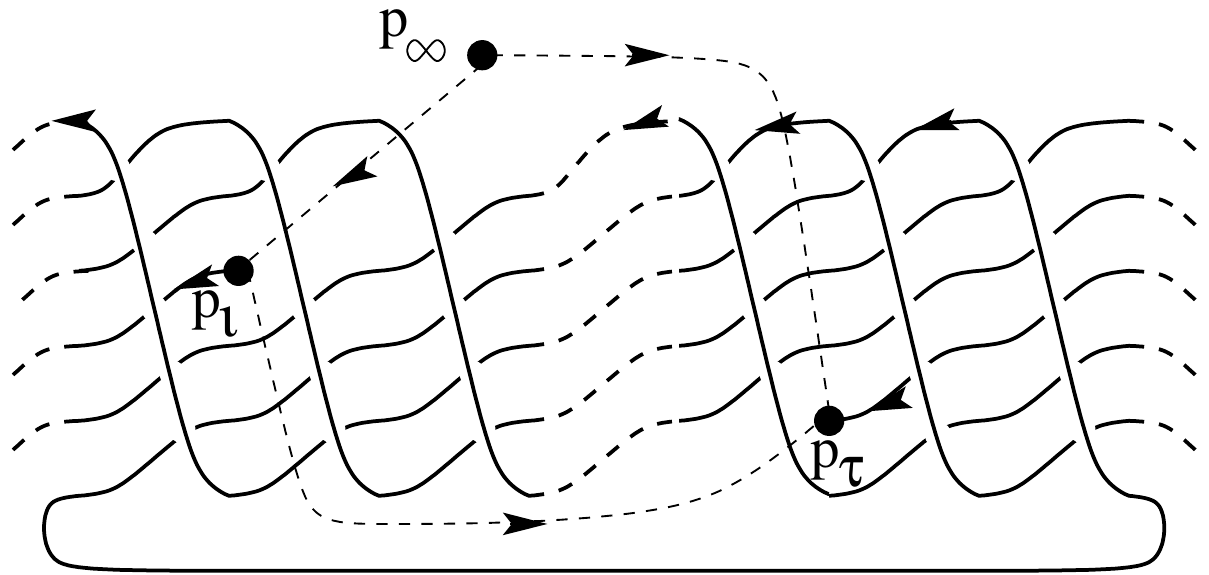}
\end{center}
\caption{}
\label{T6n}
\end{figure}

\begin{definition}
{\rm
For a planar knotoid $(K, R)$, 
define an element  
$$\Phi_\phi^{(a,b,c)} (K, R, \gamma_h )  = \sum_{ {\cal C} \in {\rm Col}^{(a,b,c)}_X (K) }
B_\phi ({\cal C}, \gamma_h)   \in \Z[A]$$
in the group ring $\Z[A]$.
Also define 
$$\Phi_\phi(K, R, \gamma_h )  = \sum \Phi_\phi^{(a,b,c)} (K, R, \gamma_h ) (a \otimes b \otimes c) \in \Z A [X]. $$
Next  define the triangular 2-cocycle invariant to be the triple
$$\Phi_\phi^{(a,b,c)} (K, R) = ( \Phi_\phi^{(a,b,c)} (K, R, \gamma_\iota ) ,   \Phi_\phi^{(a,b,c)} (K, R, \gamma_{\iota \tau} )  ,\Phi_\phi^{(a,b,c) }(K, R, \gamma_\tau )  ) \in (\Z[A])^3 . $$
Also define 
$$\Phi_\phi(K, R) = ( \Phi_\phi (K, R, \gamma_\iota ) ,   \Phi_\phi(K, R, \gamma_{\iota \tau} )  ,\Phi_\phi (K, R, \gamma_\tau )  ) \in (\Z A [X]^{\otimes 3} )^3 . $$
}
\end{definition}

\begin{figure}[htb]
\begin{center}
\includegraphics[width=1in]{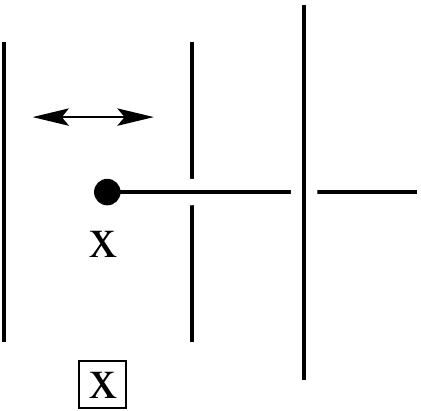}
\end{center}
\caption{}
\label{endpass}
\end{figure}

\begin{lemma}
The  triangular 2-cocycle invariant is well defined.
\end{lemma}

\begin{proof}
It is sufficient to show that the invariant values do not change under homotopy of the paths $\gamma_h$,
$h=\iota, \tau, \iota\tau$.
Generic homotopy is realized by Reidemeister moves I, II, III and the move passing through the end points of $K$. Reidemeister moves do not change the cocycle invariant values for the same reason as 
the regular 2-cocycle invariant.
When an arc goes through an end point, it happens in the region that included the end point, and the colors
of the region, the end arc, and the portion of $\gamma_h$ in the region, are all the same, before and after the move. When $\gamma_h$ goes under the end arc of $K$, see  Figure~\ref{endpass}, the contribution of the cocycle value at the crossing is trivial, since $\phi(x,x)=1$, $\phi$ being a quandle 2-cocycle.
Hence the claim follows.
\end{proof}

\begin{figure}[htb]
\begin{center}
\includegraphics[width=4.5in]{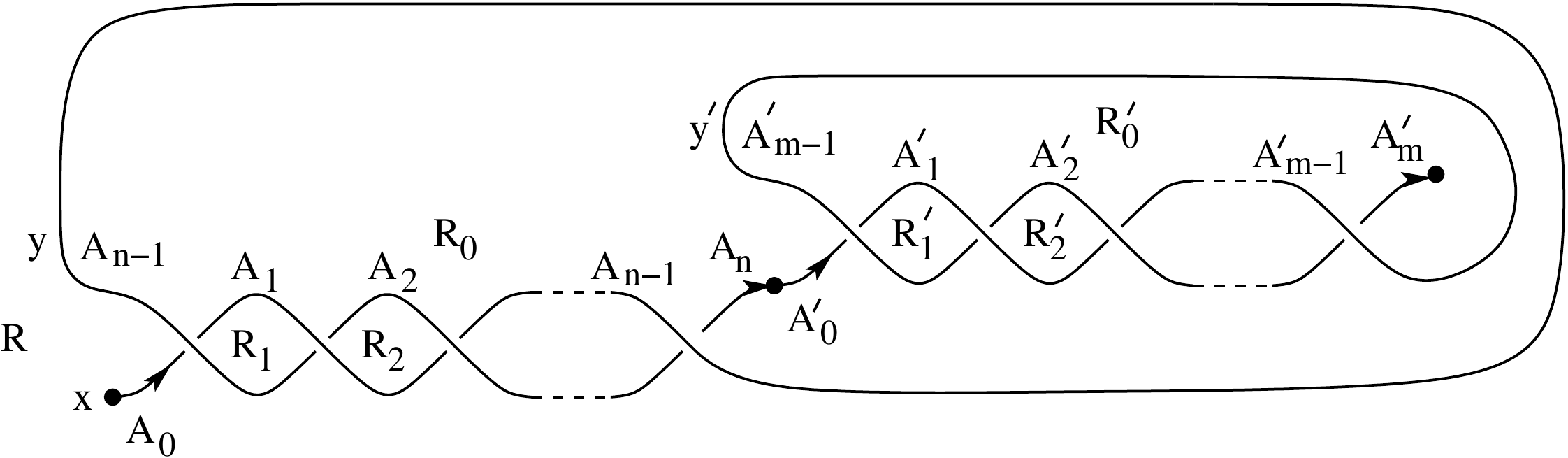}
\end{center}
\caption{}
\label{two}
\end{figure}

The triangular quandle cocycle invariant provides a stronger obstruction than end specified quandle
colorings, in the following sense.

\begin{proposition}
There exists a quandle $X$ such that  there is an infinite family of spherical knotoids with distinct planar
equivalence classes, that are distinguished by triangular cocycle invariants of $X$, while they cannot be distinguished by colorings only.
\end{proposition}

\begin{proof}
Set $X=\Z_2[t]/(t^2+t+1)$,  a four-element Alexander quandle denoted by $S_4$ in \cite{CJKLS}, also called 
the tetrahedral quandle, that corresponds to rotations of a regular tetrahedron.
We represent its elements by $\{ 0, 1, t, 1+t \} $.
We use the 2-cocycle of $X=\Z_2[t]/(t^2+t+1)$ with coefficient group $A=\Z_2$ given in \cite{CJKS},
$\phi=\prod \chi_{(a,b)}$, where the product ranges over all $a, b \in \{0, 1, 1+t \}$ such that $a \neq b$, and $\chi$ denotes Kronecker's delta.

In Figure~\ref{two}, the knotoid product  $T_{n,m}:=T_n \cdot T_m$ is depicted.
The arcs and regions are labeled as in Example~\ref{ex:Tn} for the first factor $T_n$ and primed symbols 
for the second factor $T_m$ as depicted. 
We use the computations of colorings by Alexander quandles as in Example~\ref{ex:Tn}.
Consider the colorings by $X$ specified by ${\cal C}(p_\iota)=x=0$. 
Then one computes 
$${\cal C}(A_1)=tx+(1-t)y=(1+t)y \in X, \ 
{\cal C}(A_2)=0, \ {\cal C}(A_3)=y, $$
 and this pattern repeats in period $3$.
We also compute 
$${\cal C}(R_0) = {\cal C}(R_1)=tx+(1-t)y=(1+t)y, \  
{\cal C}(R_2)=ty, $$
and starting 
${\cal C}(R_3)=x=0$, 
${\cal C}(R_4)=(1+t)y$ and 
${\cal C}(R_5)=ty$,
it repeats this pattern in period $3$.
The colorings on the part $T_m$ is similar.

There is a coloring of $T_n$ by $X$ if and only if 
$${\cal C}(A_{n-1})=y, \ {\rm and} \ {\cal C}(A_{n})={\cal C}(R_0). $$
We choose $n=4+6k$ (since the pattern repeats by 3 and $n$ must be even) for nonnegative integers $k$,  then we have 
$${\cal C}(A_{n-1})={\cal C}(A_{3+6k})=y\  {\rm  and} \ {\cal C}(A_{n})={\cal C}(A_{4+6k})=
(1+t)y={\cal C}(R_0), $$
 so that $T_n$ has a nontrivial coloring for all $y\neq 0$. 
If $y=0$, it provides a trivial coloring by $0$ on all arcs.

For the colors $x=0$ and  ${\cal C}(A_{n})={\cal C}(A_{4+6k})=
(1+t)y={\cal C}(R_0)$, we have 
$${\cal C}(A_{1}')= {\cal C}(A_{n})=x'= (1+t)y={\cal C}(R_0). $$
Again from the computations in Example~\ref{ex:Tn}, in this case, 
we obtain 
$${\cal C}(A_1')=t x'+(1+t)y'=t (1+t)y + (1+t)y'=y+(1+t)y' , \ 
{\cal C}(A_2')=x'=(1+t)y, \  {\cal C}(A_3')=y', $$
and this repeats in period 3.
We also obtain
$${\cal C}(R_1')=tx'+(1+t)y'=y+(1+t)y', \  
{\cal C}(R_2')=(1+t)x' + ty', \ 
{\cal C}(R_3')=x', $$
 and this repeats in period 3.
To obtain a coloring for $m=2$, we need 
$${\cal C}(A_2')={\cal C}(R_0'), \ {\rm and } \ {\cal C}(A_1')=y' . $$
From the above computation we have only one solution, $y'=(1+t)y$, giving a trivial coloring for the $T_m$ part. In this case 
$${\cal C}(A_2')={\cal C}(p_\tau)=(1+t)y. $$

We choose a quandle triple $({\cal C}(p_\infty), {\cal C}(p_\iota), {\cal C}(p_\tau) )=(a,b,c)=(1+t,0,1+t)$.
Then in the above computation we need $x=0$, ${\cal C}(p_\tau)={\cal C}(R_0')=(1+t)y=1+t$ so that $y=1$
to have a coloring.
Let $\phi$ be the 2-cocycle mentioned above, and let $\zeta$ be a multiplicative generator of the 
coefficient group $\Z_2$.

Let $p_\infty$ be in $R_4$. Then there is a unique coloring, 
and the cocycle value for $\Phi_\phi^{(a,b,c)} (K, R, \gamma_\iota ) $ is 
$$\phi( {\cal C}(R_4), {\cal C}(A_3))^{-1}=\phi(1+t, 1)^{-1}=\zeta^{-1}=\zeta . $$
The negative exponent of $\phi$ is due to the sign of the crossing from $R_4$ to $R$.
For other paths, we compute 
\begin{eqnarray*}
\Phi_\phi^{(a,b,c)} (K, R, \gamma_\tau ) &=& \phi( {\cal C}(R_4), {\cal C}(A_4))=\phi(1+t, 1+t)=1, \\ 
\Phi_\phi^{(a,b,c)} (K, R, \gamma_{\iota\tau} ) &=& \phi( {\cal C}(A_3), {\cal C}(R_4))^{-1} \phi ( {\cal C}(R_4), {\cal C}(A_4))
=\phi(1+t, 1)^{-1} \phi( (1+t, 1+t)=\zeta, 
\end{eqnarray*}
so that we obtain $\Phi_\phi^{(a,b,c)}(K, R_4)=(\zeta, 1, \zeta)$.

Let $p_\infty$ be in $R_1'$. Then there is a unique coloring, which is the same as above.
Since the coloring is trivial on the $T_m$ part, the cocycle value is trivial for all paths.
Hence we have $\Phi_\phi^{(a,b,c)}(K, R_4)=(1, 1, 1)$.
Since there are infinitely many choices for $n=4+6k$, the claim follows.

For both regions we verified that the number of colorings is one, therefore it does not detect planar nonequivalence.
It remains to show that the family consists of distinct spherical equivalence classes.
It is sufficient to show this for $T_n$ for various $n=4+6k$.
We use the colorings with colors of end points specified, without region colors, to
distinguish spherical equivalence.

By computations in Example~\ref{ex:Tn}, $T_n$ is colored nontrivially by the dihedral quandle ${\bf R}_\ell$ 
with colors
${\cal C}(A_0)=x$ and ${\cal C}(A_{n-1})=y$ if and only if $(n-1)(x-y)=0$. 
Thus there is a nontrivial coloring if and only if $n-1=0 \in \Z_\ell$.
This is satisfied if $n=\ell +1$ and is not satisfied if $n < \ell +1$. 
 We consider  specified colors ${\cal C}(p_\iota)=0$ and ${\cal C}(p_\tau)=\ell +1$.
 For $n=4$, we choose $\ell=3$, then $T_n$ has non-trivial coloring and $T_n$ does not for all $n>3$.
 Inductively by choosing $\ell=5+6k$ for $k \in \N$, we obtain  a spherically nonequivalent family.
\end{proof}

\section{Concluding remarks}

A natural question arises as to the interpretation of end specified quandle invariants in terms of quandle homology theory. Since the boundary colors are related to the boundary operator of second chain groups, such an interpretation is expected. Applications to quandle homology from this point of view would be an interesting topic of further research.

We also note that the invariants discussed in this paper extend naturally to \textit{linkoids} \cite{gabrovvsek23}. Linkoids are the natural generalization of knotoids in which one allows for several knotted components, either closed or with endpoints. With the exception of the triangular quandle cocylcle invariants, whose construction relies on there being exactly two endpoints, the end specified quandle invariants can be defined analogously for linkoids. Moreover the triangular quandle cocycle invariants extend to `multi-knotoids', which are linkoids with only a single open-ended component.

Consequently the scheme for deriving $\mathbb{R}$-valued entanglement measures of open curves from the terms of ${\cal Col}_X (K, R_\infty)$, as discussed in Section \ref{sec:Intro}, can be applied equally well to systems of several open curves. The resulting entanglement measures can then be used to detect the planar characteristics of the linkoids obtained from projections of these open curves. This is a feature that is not shared by the other entanglement measures studied so far, $\mathbb{R}$-valued or otherwise. Hence another direction for future research is the implementation of these entanglement measures for systems of open curves, to investigate the influence of this planar characteristic on the topology of physical examples of tangled open curve systems.
\\

\noindent
{\bf Acknowledgement} 
We would like to thank Nata\v{s}a Jonoska for organizing  Workshop: discussions of DNA knots, held at University of South Florida in 
January 22-24, 2023,   where this work was initiated. 
Thanks also go to the participants for valuable discussions.  Mohamed Elhamdadi was partially supported by Simons Foundation collaboration grant 712462.

\end{document}